\documentclass[10pt]{amsart}

\usepackage{amssymb,amsthm,amsmath}
\usepackage[numbers,sort&compress]{natbib}
\usepackage{color}
\usepackage{graphicx}
\usepackage{tikz}

\title[ ]{Some refined results  on  mixed Littlewood conjecture for pseudo-absolute values}

\author{Wencai Liu}
\address[Wencai Liu]{Department of Mathematics, University of California, Irvine, California 92697-3875, USA} \email{liuwencai1226@gmail.com}

\newcommand{\R}{\mathbb{R}}
\newcommand{\Z}{\mathbb{Z}}

\theoremstyle{plain}
\newtheorem{theorem}{Theorem}[section]
\newtheorem{corollary}[theorem]{Corollary}
\newtheorem{lemma}[theorem]{Lemma}

\theoremstyle{definition}

\newtheorem{remark}[theorem]{Remark}

\begin{document}

 \newcommand{\N}{\mathbb{N}}

\begin{abstract}
In this paper,
we study  the  mixed Littlewood conjecture with pseudo-absolute values.
For any pseudo absolute value sequence $\mathcal{D}$, we obtain the sharp criterion  such that
for almost every $\alpha$ the inequality
\begin{equation*}
    |n|_{\mathcal{D} }|n\alpha -p|\leq \psi(n)
\end{equation*}
has infinitely many coprime solutions  $(n,p)\in\N\times \Z$ for a certain one-parameter family of $\psi$. Also
under minor condition on pseudo absolute value sequences
 $\mathcal{D}_1$,$\mathcal{D}_2,\cdots, \mathcal{D}_k$,  we  obtain a sharp criterion on general
 sequence $\psi(n)$ such that
for almost every $\alpha$ the inequality
\begin{equation*}
    |n|_{\mathcal{D}_1}|n|_{\mathcal{D}_2}\cdots |n|_{\mathcal{D}_k}|n\alpha-p|\leq \psi(n)
\end{equation*}
has infinitely many coprime solutions  $(n,p)\in\N\times \Z$.


\end{abstract}
\maketitle

 \section{Introduction}

The  {\it Littlewood Conjecture }  states
that for every pair $(\alpha,\beta)$ of real numbers, we have that
\begin{equation} \label{littlewood}
\liminf_{n\to\infty}n\|n\alpha\|\|n\beta\|=0,
\end{equation}
where $||x||=\text{dist}(x,\Z)$.
   We refer the readers   to   \cite{bugeaud2014around,bugeaud2011badly}  for recent progress. By a fundamental    result of Einsiedler-Katok-Lindenstrauss  \cite{einsiedler2006invariant}  the set of pairs $(\alpha,\beta)$
for which (\ref{littlewood})  does not hold  is a zero Hausdorff dimension set.

From the
 metrical point, \eqref{littlewood} can be strengthened.
  Gallagher \cite{Gal} established   that if $\psi:\N\to\R$ is a non-negative decreasing function, then for almost every  $(\alpha,\beta)$ the inequality
\begin{equation*}\label{metricLittlewood1}
\|n\alpha\|\|n\beta\|\le \psi (n)
\end{equation*}
has infinitely    many  solutions for $n \in \N $   if and only if
$\sum_{n\in\N}\psi (n)\log n =\infty$. In
particular, it follows that
\begin{equation*}
\liminf_{n\to\infty}n \, (\log n)^2 \|n\alpha\|\|n\beta\|=0
\end{equation*}
for almost every  pair $(\alpha,\beta)$ of real numbers.
By a method of \cite{Sch}, Bugeaud and Moshchevitin\cite{bugeaud2011badly} showed  that there exist  pairs of $(\alpha,\beta)$ such that
\begin{equation*} 
\liminf_{n\to\infty}n \, (\log n)^2 \|n\alpha\|\|n\beta\|>0.
\end{equation*}
This result has been improved by Badziahin\cite{Bad}, which states that
the set of pairs $(\alpha,\beta)$ satisfying
\begin{equation*}
    \liminf_{n\to \infty} n\log n\log\log n\|n\alpha\|\|n\beta\|>0
\end{equation*}
has full Hausdorff dimension in $\R^2$.
It is conjectured that Littlewood conjecture can be strengthened to
 \begin{equation*} 
\liminf_{n\to\infty}n\log n\|n\alpha\|\|n\beta\|=0,
\end{equation*}
for all $(\alpha,\beta)\in\R^2$.

In \cite{de2004problemes},   de Mathan and Teuli\'{e} formulated  another conjecture
 --  known as the {\it Mixed Littlewood
Conjecture}.
 Let $\mathcal{D}=\{n_k\}_{k\ge 0}$ be an increasing sequence of positive integers with $n_0=1$ and $n_k|n_{k+1}$ for all $k$. We refer to such a sequence as a {\em pseudo-absolute value sequence}, and we define the   $\mathcal{D}$-adic pseudo-norm $|\cdot|_{\mathcal{D}}:\N\to \{n_k^{-1}:k\ge 0\}$ by
$$
	|{n}|_\mathcal{D} = \min\{ n_k^{-1} : n\in  n_k\Z \}.
$$
In the case  $\mathcal{D}=\{p^k\}_{k=0}^\infty$ for some integer $p\ge 2$,  we also write $|\cdot|_{\mathcal{D}}=|\cdot|_p$. 
B. de Mathan  and  O. Teuli\'{e} \cite{de2004problemes}
conjectured that for any  real number $\alpha$ and any pseudo-absolute value sequence $\mathcal{D}$, we have that
\begin{equation*}
\liminf_{n\to\infty}n|n|_\mathcal{D}\|n\alpha\|=0.
\end{equation*}
In particular, the statement that $\liminf_{n\to\infty}n|n|_p\|n\alpha\|=0$   for   every real number $\alpha$ and prime number $p$, is referred as
$p$-adic Littlewood conjecture.

 Einsiedler and  Kleinbock  have shown that any
exceptional set to the  de Mathan-Teuli\'{e} Conjecture has to be of
zero Hausdorff dimension \cite{einsiedler2007measure}.   
 By  a theorem of Furstenberg \cite{Furstenberg}, one has  that
for any two prime numbers $p,q$  and every real number $\alpha$
\begin{equation}\label{fur}
\liminf_{n\to\infty}  n |n|_{p}  |n|_{q}\|n\alpha\|  =0.
\end{equation}
This result can be made quantitative \cite{bourgain2009some}, that is
\begin{equation*}
   \liminf_{n\to\infty}  n(\log \log \log n )^{\kappa}|n|_{p}  |n|_{q}\|n\alpha\|  =0
\end{equation*}
for some $\kappa>0$.
The statement \eqref{fur} can be strengthened from a metrical point of view \cite{bugeaud2011metric}, that is,
 suppose $p_1,\ldots , p_k$ are   distinct prime numbers and
$\psi:\N\to\R$ is a non-negative decreasing function, then for
almost every real number $\alpha$ the inequality
\begin{equation*}
|n|_{p_1}\cdots |n|_{p_k}|n\alpha-p|\le \psi (n)
\end{equation*}
has infinitely   many coprime solutions $(n,p)\in \N \times\Z$  if and only if
\begin{equation}\label{impremark}
\sum_{n\in\N}(\log n)^k\psi (n)=\infty.
\end{equation}
As a corollary, it is true that
\begin{equation} \label{fur1}
\liminf_{n\to\infty}  n \, (\log n)^{k+1} |n|_{p_1}\cdots
|n|_{p_k}\|n\alpha\|  =0  \
\end{equation}
for almost every    $\alpha\in\R$. 

In \cite{harrap2013mixed},  Harrap and Haynes consider the  {\em $\mathcal{D}$-adic pseudo-absolute value}.
Given a pseudo-absolute value sequence $\mathcal{D}$  with some minor restriction, let $\mathcal{M}:\N\to\N\cup\{0\}$ be
\[\mathcal{M}(N)=\max\left\{k: n_k \leq N  \right\}.\]

Suppose that $\psi:\mathbb{N} \rightarrow \mathbb{R}$ is non-negative and decreasing and that $\mathcal{D}=\{n_k\}$ is a pseudo-absolute value sequence satisfying
	\begin{equation}\label{Euler1}
		\sum_{k=1}^{m} \frac{\varphi(n_k)}{n_k}\geq c m~\text{for all}~m\in\N ~\text{and for some}~c>0,
	\end{equation}
where $\varphi$ is the Euler phi function. Then for almost every $\alpha\in\R$ the inequality
	\begin{equation*}
		 |{n}|_\mathcal{D} |{n\alpha}-p| \le \psi(n)
	\end{equation*}
	has infinitely   many coprime solutions $(n,p) \in \mathbb{N}\times \Z$ if and only if
\begin{equation}\label{equ12}
\sum_{n=1}^{\infty} \mathcal{M}(n)\psi(n)=\infty.
\end{equation}

Note that
when $ {\mathcal{D}}=\{p^k\}$ for some positive integer $p$ we have that $\mathcal{M}(N)\asymp  \log N$. Thus  Harrap-Haynes' result implies \eqref{impremark} for $k=1$.
The first goal of this paper is to  extend \eqref{impremark} to the class of finitely many pseudo-absolute value sequences.

As pointed out in \cite{harrap2013mixed},  such generalization depends on the overlap among  pseudo-absolute value sequences.
For example\footnote{The present example and the following one are from \cite{harrap2013mixed}.} if $\mathcal{D}_1=\{2^k\}$ and $\mathcal{D}_2=\{3^k\}$, \eqref{fur1} yields that  inequality
\begin{equation*}
 |n|_{\mathcal{D}_1}|n|_{\mathcal{D}_2}\|n\alpha\|\le\psi (n)
\end{equation*}
has infinitely many solutions for almost every $\alpha$ if and only if
\[\sum_{n\in\N}(\log n)^2\psi (n)=\infty.\]
However if $\mathcal{D}_1=\mathcal{D}_2=\{2^k\}$, by \cite[Theorem 2]{bugeaud2011metric}, the inequality has infinitely many solutions for almost every $\alpha$ if and only if
\[\sum_{n\in\N} n\psi (n)=\infty.\]

Basically, the proof of \eqref{impremark} and \eqref{equ12} follows from Duffin-Schaeffer Theorem \cite{duffin1941khintchine}(see Theorem \ref{duffin}), which is a weaker version of    Duffin-Schaeffer conjecture.

 {\bf Duffin-Schaeffer Conjecture}:
 Let $\psi:\N\to \R$ be   non-negative  function and define
 \begin{equation*}
{\mathcal E}_n={\mathcal E}_n(\psi)=\bigcup_{p=1 \atop
(p,n)=1}^n\left(\frac{p-\psi(n)}{n},\frac{p+\psi(n)}{n}\right),
\end{equation*}
where $(p,n)$ is the largest common divisor between $p$ and $n$.
Then
$\lambda(\limsup{\mathcal E}_n)=1$ if
and only if $\sum_n\lambda({\mathcal E}_n)=\infty$, where $\lambda$
denotes the Lebesgue measure on $\mathbb{R}/\mathbb{Z}$.

One side of Duffin-Schaeffer conjecture is trivial. If $\sum_n\lambda({\mathcal E}_n)<\infty$, by Borel-Cantelli Lemma,
$\lambda(\limsup{\mathcal E}_n)=0$. Since it has been posted,  Duffin-Schaeffer conjecture was heavily investigated   in \cite{PV,Stra,beresnevich2013duffin,haynes2012duffin,Li1,Li2}. We should mention that  Duffin-Schaeffer conjecture is equivalent to the following statement:
Suppose $\psi:\N\to \R$ is a non-negative  function and satisfies
\begin{equation*}
    \sum_n\frac{\varphi(n)\psi(n)}{n}=\infty,
\end{equation*}
where $\varphi$ is the Euler phi function.
 Then for almost every $\alpha\in\R$ the inequality
	\begin{equation*}
		|{n\alpha}-p| \le \psi(n)
	\end{equation*}
	has infinitely   many coprime solutions $(n, p)\in \mathbb{N}\times\Z$.

We will also employ  Duffin-Schaeffer Theorem to  study mixed Littlewood conjecture in the present paper and  find a nice divergence condition for finite pseudo-absolute values.

\begin{theorem}\label{thm2}
Let $\psi:\mathbb{N} \rightarrow \mathbb{R}$  be non-negative and decreasing and let $\mathcal{D}_1=\{n_k^{1}\},\mathcal{D}_2=\{n_k^{2}\}, \cdots , \mathcal{D}_m=\{n_k^{m}\}$   be  $m$ pseudo-absolute value sequences. Suppose   $ \mathcal{D}_1, \mathcal{D}_2,\cdots,\mathcal{D}_m$  satisfies the following
condition: there exists some constant $c_1>0$ such that 
\begin{equation}\label{equ13}
		   \frac{\varphi(n_{k_1}^1n_{k_2}^2\cdots n_{k_m}^m)}{n_{k_1}^1n_{k_2}^2\cdots n_{k_m}^m}\ge c_1,
	\end{equation}
where $\varphi$ is the Euler phi function. Then for almost every $\alpha\in\R$, the inequality
	\begin{equation*}
		|{n}|_{\mathcal{D}_1 }|{n}|_{\mathcal{D}_2}\cdots |{n}|_{\mathcal{D}_m }|{n\alpha}-p| \le \psi(n)
	\end{equation*}
	has infinitely   many coprime solutions $(n, p)\in \mathbb{N}\times\Z$ if  and only if \begin{equation}\label{equ14}
\sum_{n=1}^{\infty} \frac{\psi(n)}{|{n}|_{\mathcal{D}_1 }|{n}|_{\mathcal{D}_2}\cdots |{n}|_{\mathcal{D}_m }}=\infty.
\end{equation}
\end{theorem}
\begin{remark}
Let $p_1,\ldots , p_m$ be  distinct prime numbers, and $\mathcal{D}_i=\{p^k_i\}$, $i=1,2,\cdots,m$.
For such pseudo-absolute value sequences $\mathcal{D}_i$, $i=1,2,\cdots, m$, one has \eqref{equ13} holds.
By the fact that (see \cite{bugeaud2011metric})
$$
\sum_{n\in\N}(\log n)^m\psi (n) = \infty   \quad
\Longleftrightarrow \quad \sum_{n\in\N}\frac{\psi (n)}{ |n|_{p_1}
\cdots  |n|_{p_m}} = \infty \, ,
$$
 Theorem \ref{thm2} implies \eqref{impremark}.
\end{remark}
We say a pseudo-absolute value sequence $\mathcal{D}=\{n_k\}$ is generated by finite integers if there exist prime numbers  $p_1,p_2,\cdots, p_N $ such that
every $n_k$  can be written as  $p_1^{k_1}p_2^{k_2}\cdots p_N^{k_N}$ for some proper positive integers $k_1,k_2,\cdots,k_N$. We call $p_1,p_2,\cdots,p_N$ the generators of $\mathcal{D}$.
\begin{corollary}\label{cor}
Let $\psi:\mathbb{N} \rightarrow \mathbb{R}$  be non-negative and decreasing and let $\mathcal{D}_1=\{n_{k}^1\},\mathcal{D}_2=\{n_{k}^2\}, \cdots , \mathcal{D}_m=\{n_{k}^m\}$  be  $m$ pseudo-absolute value sequences. Suppose  each $\mathcal{D}_1 ,\mathcal{D}_2 , \cdots ,\mathcal{D}_m$
is generated by finite integers. Then for almost every $\alpha\in\R$ the inequality
	\begin{equation*}
		  |{n}|_{\mathcal{D}_1 }|{n}|_{\mathcal{D}_2}\cdots |{n}|_{\mathcal{D}_m }|{n\alpha}-p| \le \psi(n)
	\end{equation*}
	has infinitely   many coprime solutions $(n,p) \in \mathbb{N}\times\Z$ if  and only if
\begin{equation*}
\sum_{n=1}^{\infty} \frac{\psi(n)}{|{n}|_{\mathcal{D}_1 }|{n}|_{\mathcal{D}_2}\cdots |{n}|_{\mathcal{D}_m }}=\infty.
\end{equation*}
\end{corollary}
\begin{proof}
If  $\mathcal{D}_j$
is generated by finite integers for each $j=1,2,\cdots,m$, one has
\eqref{equ13} holds. Thus Corollary \ref{cor} directly follows  from Theorem \ref{thm2}.
\end{proof}
Suppose there is no intersection between the pseudo-absolute value sequences. Then we can get better results.
We say two pseudo-absolute value sequences $\mathcal{D}_1=\{n_{k}^1\}$  and  $\mathcal{D}_2=\{n_{k}^2\}$ are coprime  if
$n_{i}^1$ and $n_{j}^2$ are coprime for any $i,j\in \N$.

\begin{theorem}\label{thmnew}
Let $\psi:\mathbb{N} \rightarrow \mathbb{R}$  be non-negative and decreasing.
Suppose the pseudo-absolute value sequences $\mathcal{D}_1=\{n_{k}^1\},\mathcal{D}_2=\{n_{k}^2\}, \cdots , \mathcal{D}_m=\{n_{k}^m\}$ are mutually coprime
and
\begin{equation}\label{equ13new}
		\sum_{n_{k_1}^1n_{k_2}^2\cdots n_{k_m}^m\leq N}  \frac{\varphi(n_{k_1}^1n_{k_2}^2\cdots n_{k_m}^m)}{n_{k_1}^1n_{k_2}^2\cdots n_{k_m}^m}\ge c_2\#\{(k_1,k_2,\cdots,k_m):n_{k_1}^1n_{k_2}^2\cdots n_{k_m}^m\leq N\},
	\end{equation}
for some constant $c_2>0$.
Suppose that there exists some $c_3$ with  $0<c_3<1$ such that   
\begin{equation}\label{1829equ1}
  \sum_{ n_{k_1}^1n_{k_2}^2\cdots n_{k_m}^m\leq N} n_{k_1}^1n_{k_2}^2\cdots n_{k_m}^m\leq c_3 N \#\{(k_1,k_2,\cdots,k_m):n_{k_1}^1n_{k_2}^2\cdots n_{k_m}^m\leq N\},
\end{equation}
for all large $N$.

 Then for almost every $\alpha\in\R$,  the inequality
	\begin{equation*}
		  |{n}|_{\mathcal{D}_1 }|{n}|_{\mathcal{D}_2}\cdots |{n}|_{\mathcal{D}_m }|{n\alpha}-p| \le \psi(n)
	\end{equation*}
	has infinitely   many coprime solutions $(n,p) \in \mathbb{N}\times \Z$ if  and only if
\begin{equation*}
\sum_{n=1}^{\infty} \psi(n) \#\{(k_1,k_2,\cdots,k_m):n_{k_1}^1n_{k_2}^2\cdots n_{k_m}^m\leq n\}=\infty.
\end{equation*}
\end{theorem}
Duffin-Schaeffer theorem is crucial to the proof of Theorems \ref{thm2} and \ref{thmnew}. However
Duffin-Schaeffer theorem requires good match between sequence $\psi(n)$ and Euler function $\varphi(n)$, so that hypotheses  \eqref{Euler1}, \eqref{equ13} and \eqref{equ13new} are very important.  For some nice functions $\psi(n)$,
Duffin-Schaeffer theorem can be improved \cite{beresnevich2013duffin,haynes2012duffin,Li1,Li2}.
  We will use  \cite[Theorem 1.17]{Li2} to study the mixed Littlewood conjecture and find that restriction \eqref{Euler1} is not necessary in some sense.

   Given $n\in\N$ and $x\in\R$, define
   \begin{equation*}
    ||nx||^\prime=\min\{|nx-p|:p\in\Z, (n,p)=1\}.
   \end{equation*}

 \begin{theorem}\label{thmnewhar}
Let $\mathcal{D}=\{n_k\}$ be a pseudo-absolute value sequence and define
\begin{equation}\label{GMhar}
		\mathfrak{M}(n)=\sum_{n_{k } \leq n}  \frac{\varphi(n_{k})}{n_k}.
	\end{equation}
Suppose $\epsilon\geq0$.
Then   for almost every $\alpha\in\R$
	\begin{equation}\label{Ghar1}
	\liminf_{n\to\infty} n \mathfrak{M}(n)(\log n)^{1+\epsilon}	  |{n}|_{\mathcal{D}  } ||{n\alpha}||^{\prime} = 0,
	\end{equation}
if and only if $\epsilon=0$.


\end{theorem}

\section{Proof of Theorem \ref{thm2}}
In this paper, we always assume $C$ ($c$) is a large (small) constant, which is different even in the same equation.
We should mention that the constant $C$ ($c$) also  depends on $c_1,c_2$ and $c_3$ in the Theorems.

Before we give the proof of  Theorem \ref{thm2}, some preparations are necessary.
\begin{lemma}\cite[Lemma 2]{bugeaud2011metric}\label{lem1}
Let  $p_1,\ldots , p_k$ be  distinct prime numbers and  $N\in\N$.
Then
\begin{equation*}
\sum_{\substack{n\le N\\p_1,\ldots , p_k\nmid n}}\frac{\varphi
(n)}{n}=\frac{6N}{\pi^2}\prod_{i=1}^k
\frac{p_i}{p_i+1}+O\left(\log N\right).
\end{equation*}
\end{lemma}

Obviously, Lemma \ref{lem1} implies the following lemma.
\begin{lemma}\label{keysection2}
Suppose $d_1,d_2,\cdots, d_m\geq 2$. Then there exists some $d>0$ only depending on $m$ such that
\begin{equation*}
\sum_{\substack{n=1\\d_1\nmid n,d_2\nmid n,\cdots d_m\nmid n}}^N\frac{\varphi (n)}{n}\ge dN\quad\text{for any }~ N\in\N  .
\end{equation*}
\end{lemma}
%
%
%
\begin{theorem}[Duffin-Schaeffer \cite{duffin1941khintchine}]\label{duffin}
Suppose $\sum_{n=1}^{ \infty}  \psi (n)=\infty$ and
\begin{equation*}
 \limsup_{N\to\infty}\left(\sum_{n=1}^{ N} \frac{\varphi (n)}{n}\psi (n)\right)  \left(\sum_{n=1}^{ N}  \psi (n) \right)^{-1} \, >  \, 0  \ .
\end{equation*}
Then for almost every $\alpha$, the inequality
\begin{equation*}
    |n\alpha-p|\leq \psi (n)
\end{equation*}
has infinitely  many coprime solutions $(n,p)\in\N\times\Z$.
\end{theorem}

Suppose $\mathcal{D}_1=\{n_{k}^1\},\mathcal{D}_2=\{n_{k}^2\}, \cdots ,\mathcal{D}_m=\{n_{k}^m\}$  are  $m$ pseudo-absolute value sequences.
Denote $d_{k+1}^j=\frac{n_{k+1}^j}{n_{k}^j}$ for $j=1,2,\cdots,m$.
Define  a  subset  $S(n)$  of $\N^m$ as follows:
\begin{equation*}
    S(n)=\{(k_1,k_2,\cdots,k_m):(k_1,k_2,\cdots,k_m)\in \N^m \text{ and }  \text{lcm}(n_{k_1}^1,n_{k_2}^2,\cdots ,n_{k_m}^m) \leq n\},
\end{equation*}
where  lcm($k_1,k_2\cdots,k_m$) means the least common  multiple number of $k_1,k_2\cdots,k_m$.
For any $(k_1,k_2,\cdots,k_m)\in S(n)$,
we  define $f(n;k_1,k_2,\cdots,k_m)\in \N$ as the largest positive integer  such that
\begin{equation*}
    \text{lcm}(n_{k_1}^1,n_{k_1}^2,\cdots ,n_{k_m}^m)f(n;k_1,k_2,\cdots,k_m)\leq n.
\end{equation*}

\par
{\bf Proof of Theorem \ref{thm2}.}

Without of loss of generality, assume $\alpha\in[0,1)$.
Define
\begin{equation*}
    {\mathcal E}_n={\mathcal E}_n(\psi_0)=\bigcup_{p=1 \atop
(p,n)=1}^n\left(\frac{p-\psi_0(n)}{n},\frac{p+\psi_0(n)}{n}\right),
\end{equation*}
where
\begin{equation*}
    \psi_0(n)= \frac{  \psi(n)}{ |{n}|_{\mathcal{D}_1}|{n}|_{\mathcal{D}_2}\cdots |n|_{\mathcal{D}_m}}.
\end{equation*}
The Lebesgue measure of ${\mathcal E}_n$ is obviously bounded above
by $\frac{2\psi_0(n)}{n}\varphi(n)$.
Obviously, coprime pair $(n,p)\in\N\times\Z$ is a solution of $|n\alpha-p|\leq \psi_0(n)$ if and only if
$\alpha\in  {\mathcal E}_n$.

If
\begin{equation*}
    \sum_{n=1}^{\infty} \frac{  \psi(n)}{ |{n}|_{\mathcal{D}_1}|{n}|_{\mathcal{D}_2}\cdots |n|_{\mathcal{D}_m}}< \infty,
\end{equation*}
one has
\begin{equation}\label{Gmeasure1}
   \sum_n \lambda( {\mathcal E}_n)<\infty.
\end{equation}
By  Borel-Cantelli Lemma,  the inequality
	\begin{equation*}
	  |{n}|_{\mathcal{D}_1 }|{n}|_{\mathcal{D}_2}\cdots |{n}|_{\mathcal{D}_m }|{n\alpha}-p| \le \psi(n)
	\end{equation*}
	has infinitely   many solutions $(n,p) \in \mathbb{N}\times \Z$ only for a  zero Lebesgue measure set of $\alpha$.

Now we start to prove the other side.
First, one has
\begin{equation*}
    \sum_{n=1}^N \frac{\varphi(n)\psi(n)}{n|{n}|_{\mathcal{D}_1}|{n}|_{\mathcal{D}_2}\cdots |n|_{\mathcal{D}_m}}
\end{equation*}
\begin{equation}\label{equ11}
	   =    \sum_{n=1}^N
	\left(\psi(n)-\psi(n+1) \right) \sum_{j=1}^n \frac{\varphi(j)}{j|{j}|_{\mathcal{D}_1}|{j}|_{\mathcal{D}_2}\cdots |j|_{\mathcal{D}_m}}
	+ \psi(N+1)\sum_{j=1}^N  \frac{\varphi(j)}{j|{j}|_{\mathcal{D}_1}|{j}|_{\mathcal{D}_2}\cdots |j|_{\mathcal{D}_m}}.
\end{equation}
Now we are in the position to estimate the inner sums. Direct computation implies

$$\sum_{j=1}^n \frac{\varphi(j)}{ j|{j}|_{\mathcal{D}_1}|{j}|_{\mathcal{D}_2}\cdots |j|_{\mathcal{D}_m}}\;\;\;\;\;\;\;\;\;\;\;\;\;\;\;\;\;\;\;\;\;\;\;\;\;\;\;\;\;\;\;\;\;\;\;\;\;\;\;\;\;\;\;\;\;\;\;\;\;\;\;\;\;\;\;\;\;\;\;\;\;\;\;\;\;\;\;\;\;\;\;\;$$
\begin{align}
&=\sum_{(k_1,k_2,\cdots,k_m)\in S(n)}\sum_{\substack{j=1\\n_{k_1}^1|j,n_{k_2}^2|j,\cdots,n_{k_m}^m|j\\~n_{k_1+1}^1\nmid j,n_{k_2+1}^2\nmid j,\cdots,n_{k_m+1}^m\nmid j}}^n\frac{\varphi(j)}{ j|{j}|_{\mathcal{D}_1}|{j}|_{\mathcal{D}_2}\cdots |j|_{\mathcal{D}_m}}\\
&= \sum_{(k_1,k_2,\cdots,k_m)\in S(n)}\frac{n_{k_1}^1n_{k_2}^2\cdots n_{k_m}^m}{\text{lcm}(n_{k_1}^1,n_{k_2}^2,\cdots ,n_{k_m}^m)}\sum_{\substack{1\le j\le f(n;k_1,k_2,\cdots,k_m)\\d_{k_1+1}^1\nmid j,d_{k_2+1}^2\nmid j,\cdots d_{k_m+1}^m\nmid j}}\frac{\varphi(\text{lcm} (n_{k_1}^1,n_{k_2}^2,\cdots ,n_{k_m}^m) j)}{j}\nonumber\\
&\ge\sum_{(k_1,k_2,\cdots,k_m)\in S(n)} n_{k_1}^1n_{k_2}^2\cdots n_{k_m}^m\frac{\varphi(\text{lcm}(n_{k_1}^1,n_{k_2}^2,\cdots, n_{k_m}^m ))}{\text{lcm}(n_{k_1}^1,n_{k_2}^2,\cdots, n_{k_m}^m )}\sum_{\substack{1\le j\le f(n;k_1,k_2,\cdots,k_m)\\d_{k_1+1}^1\nmid j,d_{k_2+1}^2\nmid j,\cdots d_{k_m+1}^m\nmid j}}\frac{\varphi(j)}{j}\nonumber\\
&\ge c\sum_{(k_1,k_2,\cdots,k_m)\in S(n)}f(n;k_1,k_2,\cdots,k_m) \varphi(n_{k_1}^1n_{k_2}^2\cdots n_{k_m}^m ),\label{equ1}
\end{align}
where the first inequality holds by the fact that $\varphi(mn)\geq \varphi(m)\varphi(n)$ and the second inequality holds by Lemma \ref{keysection2} and the fact that
\begin{equation*}
   \frac{\varphi(\text{lcm}(n_{k_1}^1,n_{k_2}^2,\cdots, n_{k_m}^m ))}{\text{lcm}(n_{k_1}^1,n_{k_2}^2,\cdots, n_{k_m}^m )}=\frac{\varphi( n_{k_1}^1n_{k_2}^2\cdots n_{k_m}^m )}{n_{k_1}^1n_{k_2}^2\cdots n_{k_m}^m }.
\end{equation*}

By   \eqref{equ13} and \eqref{equ1}, we get
\begin{eqnarray}\label{finall2}
  \sum_{j=1}^n \frac{\varphi(j)}{ j|{j}|_{\mathcal{D}_1}|{j}|_{\mathcal{D}_2}\cdots |j|_{\mathcal{D}_m}}
   &\geq& c\sum_{(k_1,k_2,\cdots,k_m)\in S(n)}n_{k_1}^1n_{k_2}^2 \cdots  n_{k_m}^mf(n;k_1,k_2,\cdots,k_m) .
\end{eqnarray}
One the other hand, we have
\begin{align}
\sum_{j=1}^n\frac{1}{|j|_{\mathcal{D}_1}|j|_{\mathcal{D}_2}\cdots |j|_{\mathcal{D}_m}}&=\sum_{(k_1,k_2,\cdots,k_m)\in S(n)} n_{k_1}^1n_{k_2}^2\cdots n_{k_m}^m\sum^n_{\substack{j=1\\n_{k_1}^1|j,n_{k_2}^2|j,\cdots,n_{k_m}^m|j\\~n_{k_1+1}^1\nmid j,n_{k_2+1}^2\nmid j,\cdots,n_{k_m+1}^m\nmid j}}1\label{equnew}\\
&\leq \sum_{(k_1,k_2,\cdots,k_m)\in S(n)}n_{k_1}^1n_{k_2}^2 \cdots  n_{k_m}^mf(n;k_1,k_2,\cdots,k_m).\label{finall1}
\end{align}
Finally, putting \eqref{finall2} and \eqref{finall1} together,  we get
\begin{equation*}
    \sum_{j=1}^n \frac{\varphi(j)}{ j|{j}|_{\mathcal{D}_1}|{j}|_{\mathcal{D}_2}\cdots |j|_{\mathcal{D}_m}} \geq c\sum_{j=1}^n\frac{1}{|j|_{\mathcal{D}_1}|j|_{\mathcal{D}_2}\cdots |j|_{\mathcal{D}_m}}.
\end{equation*}
Combining with  \eqref{equ11},
we have
\begin{eqnarray*}
   \sum_{n=1}^N \frac{\varphi(n)\psi(n)}{n|{n}|_{\mathcal{D}_1}|{n}|_{\mathcal{D}_2}\cdots |n|_{\mathcal{D}_m}} &\geq& \sum_{n=1}^N
	\left(\psi(n)-\psi(n+1) \right) \sum_{j=1}^n \frac{c}{|{j}|_{\mathcal{D}_1}|{j}|_{\mathcal{D}_2}\cdots |j|_{\mathcal{D}_m}}
	+ \psi(N+1)\sum_{j=1}^N  \frac{c}{|{j}|_{\mathcal{D}_1}|{j}|_{\mathcal{D}_2}\cdots |j|_{\mathcal{D}_m}} \\
   &\geq& c \sum_{n=1}^N \frac{ \psi(n)}{ |{n}|_{\mathcal{D}_1}|{n}|_{\mathcal{D}_2}\cdots |n|_{\mathcal{D}_m}}.
\end{eqnarray*}
Now Theorem \ref{thm2} follows from \eqref{equ14} and Theorem \ref{duffin}.

\section{Proof of Theorem \ref{thmnew}}
The proof of Theorem \ref{thmnew} is similar to the proof of Theorem \ref{thm2} or
\eqref{equ12}. We need one lemma first.
Denote
\begin{equation*}
    \mathcal{M}(n)=\#\{(k_1,k_2,\cdots,k_m):n_{k_1}^1n_{k_2}^2\cdots n_{k_m}^m\leq n\}-1.
\end{equation*}
\begin{lemma}\label{Le18}
Under the conditions of Theorem \ref{thmnew},
the following estimate holds,
\begin{equation}\label{1829equ2}
  N\mathcal{M}(N)\asymp \sum_{n=1}^N \mathcal{M}(n).
\end{equation}
\end{lemma}
\begin{proof}
It suffices to show that
\begin{equation*}
   N\mathcal{M}(N)\leq O(1) \sum_{n=1}^N \mathcal{M}(n).
\end{equation*}

We rearrange  $n_{k_1}^1 n_{k_2}^2\cdots n_{k_m}^m$ as a monotone sequence $t_0=1,t_1,t_2,\cdots, t_k\cdots$. Then, we have
\begin{eqnarray}
  \sum_{n=1}^N \mathcal{M}(n) &=& \sum_{k=0}^{\mathcal{M}(N)-1} k (t_{k+1}-t_k)+\mathcal{M}(N)(N-t_{\mathcal{M}(N)}+1)\nonumber\\
   &=& (N+1)\mathcal{M}(N)-  \sum_{k=0}^{\mathcal{M}(N)}t_k.\label{Gmar231}
\end{eqnarray}
By the assumption \eqref{1829equ1}, one has
\begin{equation}\label{Gmar232}
 \sum_{k=0}^{\mathcal{M}(N)}t_k\leq c_3 N\mathcal{M}(N),
\end{equation}
for some $0<c_3<1$.

Now the Lemma follows from
  \eqref{Gmar231} and \eqref{Gmar232}.
\end{proof}
{\bf Proof of Theorem \ref{thmnew}}

We employ the same notations as in the proof of Theorem \ref{thm2}.

By the fact that the pseudo-absolute value sequences are mutually coprime, one has
\begin{equation*}
    \mathcal{M}(n)+1=\#S(n).
\end{equation*}
Moreover,
\begin{equation*}
  \frac{n}{2} \leq  n_{k_1}^1 n_{k_2}^2\cdots  n_{k_m}^mf(n;k_1,k_2,\cdots,k_m)\leq n.
\end{equation*}
By \eqref{equ1} and assumption \eqref{equ13new}, we have
\begin{align}
\sum_{j=1}^n \frac{\varphi(j)}{ j|{j}|_{\mathcal{D}_1}|{j}|_{\mathcal{D}_2}\cdots |{j}|_{\mathcal{D}_m}} &\geq c\sum_{(k_1,k_2,\cdots,k_m)\in S(n)}f(n;k_1,k_2,\cdots,k_m) \varphi(n_{k_1}^1n_{k_2}^2\cdots n_{k_m}^m )\nonumber\\
&\ge c n \sum_{(k_1,k_2,\cdots,k_m)\in S(n)}\frac{ \varphi(n_{k_1}^1n_{k_2}^2\cdots n_{k_m}^m )}{n_{k_1}^1n_{k_2}^2\cdots n_{k_m}^m}\nonumber\\
&\ge cn\mathcal{M}(n).\label{equ1new1wen3}
\end{align}
By  \eqref{equ1new1wen3} and \eqref{equnew}, we have
\begin{equation}\label{equnew1}
  cn\mathcal{M}(n)\leq  \sum_{j=1}^n\frac{1}{|j|_{\mathcal{D}_1}|j|_{\mathcal{D}_2}\cdots |j|_{\mathcal{D}_m}} \leq n\mathcal{M}(n).
\end{equation}
Suppose
\begin{equation*}
    \sum_{n} \psi(n)\mathcal{M}(n)<\infty.
\end{equation*}
In this case, by \eqref{1829equ2}, one has
\begin{equation*}
    \sum_{n=1}^N \frac{\psi(n)}{|{n}|_{\mathcal{D}_1}|{n}|_{\mathcal{D}_2}\cdots |n|_{\mathcal{D}_m}}\;\;\;\;\;\;\;\;\;\;\;\;\;\;\;\;\;\;\;\;\;\;\;\;\;\;\;\;\;\;\;\;\;\;\;\;
    \;\;\;\;\;\;\;\;\;\;\;\;\;\;\;\;\;\;\;\;\;\;\;\;\;\;\;\;\;\;\;\;\;\;\;\;\;\;\;\;\;\;\;\;\;\;\;\;
\end{equation*}
\begin{eqnarray}
    &=&  \sum_{n=1}^N
	\left(\psi(n)-\psi(n+1) \right) \sum_{j=1}^n \frac{1}{|{j}|_{\mathcal{D}_1}|{j}|_{\mathcal{D}_2}\cdots |j|_{\mathcal{D}_m}}
	+ \psi(N+1)\sum_{j=1}^N  \frac{ 1}{|{j}|_{\mathcal{D}_1}|{j}|_{\mathcal{D}_2}\cdots |j|_{\mathcal{D}_m}} \nonumber\\
   &\leq &  \sum_{n=1}^N
	\left(\psi(n)-\psi(n+1) \right) n\mathcal{M}(n) +\psi(N+1)N\mathcal{M}(N)\nonumber\\
  &\leq & C \sum_{n=1}^N \left(\psi(n)-\psi(n+1) \right) \sum_{j=0}^{n}\mathcal{M}(j) +\psi(N+1)N\mathcal{M}(N)\nonumber\\
    &\leq& C \sum_{n=1}^N \psi(n)\mathcal{M}(n)<\infty ,\label{equ6new}
\end{eqnarray}
where the first inequality holds by \eqref{equnew1}.
By  Borel-Cantelli Lemma,  the inequality
	\begin{equation*}
	  |{n}|_{\mathcal{D}_1 }|{n}|_{\mathcal{D}_2}\cdots |{n}|_{\mathcal{D}_m }|{n\alpha}-p| \le \psi(n)
	\end{equation*}
	has infinitely   many coprime solutions $(n,p) \in \mathbb{N}\times\Z$ only for a zero Lebesgue  measure set of $\alpha$.

Now we are in the position to  prove the other side.

Suppose
\begin{equation*}
    \sum_{n} \psi(n)\mathcal{M}(n)=\infty.
\end{equation*}

By \eqref{equ11} and \eqref{equ1new1wen3},
one has
\begin{eqnarray}
   \sum_{n=1}^N \frac{\varphi(n)\psi(n)}{n|{n}|_{\mathcal{D}_1}|{n}|_{\mathcal{D}_2}\cdots |n|_{\mathcal{D}_m}} &\geq &  c \sum_{n=1}^N
	\left(\psi(n)-\psi(n+1) \right) nM(n) +c\psi(N+1)N\mathcal{M}(N) \nonumber \\
   &\geq &c \sum_{n=1}^N\psi(n)\mathcal{M}(n).\label{Gwen4}
\end{eqnarray}
Thus
\begin{equation}\label{equ7new}
     \sum_{n=1}^{\infty} \frac{\psi(n)}{|{n}|_{\mathcal{D}_1}|{n}|_{\mathcal{D}_2}\cdots |n|_{\mathcal{D}_m}} =\infty.
\end{equation}
By \eqref{equ6new} and \eqref{Gwen4}, we have

\begin{equation}\label{equ8new}
    \sum_{n=1}^N \frac{\varphi(n)\psi(n)}{n|{n}|_{\mathcal{D}_1}|{n}|_{\mathcal{D}_2}\cdots |n|_{\mathcal{D}_m}}\geq c \sum_{n=1}^N\frac{\psi(n)}{|n|_{\mathcal{D}_1}|n|_{\mathcal{D}_2}\cdots |n|_{\mathcal{D}_m}}.
\end{equation}
Applying \eqref{equ7new} and \eqref{equ8new} to Theorem \ref{duffin},
we finish the proof.
\section{Proof of Theorem \ref{thmnewhar}}

Before we give the proof,  one lemma is necessary.
\begin{lemma}\label{Lemar23}
Let $\mathcal{D}=\{n_k\}$ be a pseudo-absolute value sequence and  $\mathfrak{M}(n)$ be given by \eqref{GMhar}.
We have the following estimate,
\begin{equation}\label{Gmar231829equ2}
  N\mathfrak{M}(N)\asymp \sum_{n=1}^N \mathfrak{M}(n).
\end{equation}
\end{lemma}
\begin{proof}
 It is easy to see that \eqref{Gmar231829equ2}  holds if  sequence $\mathfrak{M}(n)$ is bounded.
 Thus, we assume  $\mathfrak{M}(n)\to \infty$ as $n\to \infty$.

It suffices to show that
\begin{equation*}
   N\mathfrak{M}(N)\leq O(1) \sum_{n=1}^N \mathfrak{M}(n).
\end{equation*}
As usual, let $ \mathcal{M}(N)$ be the largest $k$ such that $n_k\leq N$.

By the definition of $\mathfrak{M}(n)$, one has
\begin{eqnarray}
  \sum_{n=1}^N \mathfrak{M}(n) &=& \sum_{k=0}^{\mathfrak{M}(N)} \left(\sum_{j=0}^k\frac{\varphi(n_j)}{n_j}\right) (n_{k+1}-n_k)+\left(\sum_{j=0}^{ \mathcal{M}(N)}\frac{\varphi(n_j)}{n_j}\right)(N-n_{\mathcal{M}(N)}+1)\nonumber\\
   &=& (N+1)\left(\sum_{j=0}^{ \mathcal{M}(N)}\frac{\varphi(n_j)}{n_j}\right) -  \sum_{k=0}^{  \mathcal{M}(N)}n_k\frac{\varphi(n_k)}{n_k}\nonumber\\
   &=& (N+1)\mathfrak{M}(N)-  \sum_{k=0}^{  \mathcal{M}(N)} \varphi(n_k).\label{Gmar233}
\end{eqnarray}
By the  fact that $n_{k+1}\geq 2n_k$, one has
\begin{equation*}
  \sum_{k=0}^{\mathcal{M}(N)}n_k\leq N\sum_{k=0}^{\mathcal{M}(N)}\frac{1}{2^k}\leq 2N.
\end{equation*}
This implies
\begin{equation}\label{Gmar234}
  \sum_{k=0}^{  \mathcal{M}(N)} \varphi(n_k)\leq 2N.
\end{equation}
By \eqref{Gmar233} and \eqref{Gmar234}, we have
\begin{equation*}
  N\mathfrak{M}(N)\leq O(1) \sum_{n=1}^N \mathfrak{M}(n)
\end{equation*}
We finish the proof.
\end{proof}
We will split the proof Theorem \ref{thmnewhar} into two parts.
\begin{theorem}\label{Lithm1}
Let $\mathcal{D}=\{n_k\}$ be a pseudo-absolute value sequence and  $\mathfrak{M}(n)$ be given by \eqref{GMhar}.  Suppose  $\psi:\mathbb{N}\rightarrow\mathbb{R}^+$ is non-increasing and
\begin{equation}\label{GLi3}
  \sum_{n}\  \psi(n) \mathfrak{M}(n) <\infty.
\end{equation}
Then for almost every $\alpha$, the inequality
\begin{equation*}
    |{n}|_{\mathcal{D}}|n\alpha-p|\leq \psi (n)
\end{equation*}
has finitely  many coprime solutions $(n,p)\in\N\times\Z$.
In particular, for any $\epsilon>0$,
\begin{equation*}
    \liminf_{n\to\infty} n \mathfrak{M}(n)(\log n)^{1+\epsilon}	  |{n}|_{\mathcal{D}  } ||{n\alpha}||^{\prime} = 0
\end{equation*}
holds for a zero Lebesgue measure  set $\alpha\in\R$.
\end{theorem}
\begin{proof}
The proof of Theorem \ref{Lithm1} is based on Borel-Cantelli Lemma.
Without  loss of generality, assume $\alpha\in[0,1)$.
Define
\begin{equation*}
    {\mathcal E}_n={\mathcal E}_n(\psi_0)=\bigcup_{p=1 \atop
(p,n)=1}^n\left(\frac{p-\psi_0(n)}{n},\frac{p+\psi_0(n)}{n}\right),
\end{equation*}
where
\begin{equation*}
    \psi_0(n)= \frac{  \psi(n)}{ |{n}|_{\mathcal{D}}}.
\end{equation*}
By the proof of Theorem \ref{thm2},
in order to prove Theorem \ref{Lithm1}, we only need to show
\begin{equation*}
    \sum_{n} \lambda( {\mathcal E}_n)<\infty.
\end{equation*}
Like \eqref{equ11}, one has
\begin{equation}\label{eqnli4}
	\sum_{n=1}^N \frac{\varphi(n)\psi(n)}{n|{n}|_{\mathcal{D}}} \,  = \,  \sum_{n=1}^N
	\left(\psi(n)-\psi(n+1) \right) \sum_{m=1}^n \frac{\varphi(m)}{m|{m}|_{\mathcal{D}}}
	+ \psi(N+1)\sum_{m=1}^N \frac{\varphi(m)}{m|{m}|_{\mathcal{D}}}.
\end{equation}
We estimate the inner sums here (denote $d_{k+1}=n_{k+1}/n_k$) by
\begin{align*}
\sum_{m=1}^n \frac{\varphi(m)}{m|{m}|_{\mathcal{D}}}&=\sum_{n_k\leq n}\sum_{\substack{m=1\\n_k|m,~n_{k+1}\nmid m}}^n\frac{\varphi(m)}{m|{m}|_{\mathcal{D}}}\\
&=\sum_{n_k\leq n}\sum_{\substack{1\le m\le n/n_k\\d_{k+1}\nmid m}}\frac{\varphi(n_km)}{m}\\
&\le \sum_{n_k\leq n}\varphi (n_k)\sum_{\substack{1\le m\le n/n_k\\d_{k+1}\nmid m}}1\\
&\le n\sum_{n_k\leq n}\frac{\varphi (n_k)}{n_k},\\
&=   n\mathfrak{M}(n),
\end{align*}
where the first inequality holds by the fact that
\begin{equation*}
    \varphi(nm)\leq m\varphi(n).
\end{equation*}
Therefore, by \eqref{eqnli4} and \eqref{Gmar231829equ2}, one has
\begin{eqnarray*}
  \sum_{n=1}^N\lambda({\mathcal E}_n) &\leq&   \sum_{n=1}^N\frac{2\psi_0(n)}{n}\varphi(n)\\
   &=& 2\sum_{n=1}^N \frac{\varphi(n)\psi(n)}{n|{n}|_{\mathcal{D}}}\\
   &\leq&  C \sum_{n=1}^N
	\left(\psi(n)-\psi(n+1) \right)  n\mathfrak{M}(n)+C\psi(N+1)N\mathfrak{M}(N)\\
&\leq& C \sum_{n=1}^N
	\left(\psi(n)-\psi(n+1) \right)  \sum_{j=1}^n\mathfrak{M}(j)+C\psi(N+1)N\mathfrak{M}(N)
\\
 &\leq& \sum_{n=1}^{N+1} C\psi(n)\mathfrak{M}(n).
\end{eqnarray*}
Combining with assumption \eqref{GLi3},   $\sum_{n} \lambda( {\mathcal E}_n)<\infty$ follows.
\end{proof}

The remaining  part    of Theorem \ref{thmnewhar} needs more energy to prove.
In the previous two sections, we used Duffin-Schaeffer theorem to complete the proof. Now,  we will apply   the following lemma to finish our proof.
\begin{lemma}\cite[Theorem 1.17]{Li2}\label{Har}
Let $\psi:\mathbb{N}\rightarrow\mathbb{R}$ be a
non-negative function.   Suppose
\begin{equation}\label{GLi1}
\sum_{n\in\N: G_n\geq 3}\frac{\log G_n}{n\cdot \log\log G_n}=\infty,
\end{equation}
 where
\begin{equation}\label{Gli2}
G_n=\sum_{k=2^{2^n}+1}^{2^{2^{n+1}}}\frac{\psi(k)\varphi(k)}{k}.
\end{equation}
Then for almost every $\alpha$, the  inequality
\begin{equation*}
    |n\alpha -p|\leq \psi (n)
\end{equation*}
has infinitely  many coprime solutions $(n,p)\in \N\times\Z$.
\end{lemma}
The next lemma is easy to prove by M\"obius function or follows from  Lemma \ref{lem1} ($k=1$) directly.
\begin{lemma}\label{keysection2li}
For any $d\in \N$, we have
\begin{equation*}
\sum_{\substack{n=N_1\\d\nmid n}}^{N_2}\frac{\varphi (n)}{n}\ge \max\{0, \frac{4}{\pi^2}(N_2-N_1)-O(\log N_2)\}\quad\text{for all }~ 0<N_1<N_2.
\end{equation*}
\end{lemma}
\begin{remark}
The sharp bound $\frac{4}{\pi^2}$ can be achieved when $d=2$.
\end{remark}
\begin{theorem}\label{zeroone}
Let $\psi:\N\to \R$ be   non-negative  function and $\lim_{n\to\infty}\psi(n)=0$. Define
 \begin{equation*}
{\mathcal E}_n(\psi)=\bigcup_{p=1 \atop
(p,n)=1}^n\left(\frac{p-\psi(n)}{n},\frac{p+\psi(n)}{n}\right).
\end{equation*}
Then the following claims are true.

\begin{description}
  \item[Zero-one law] $\lambda(\limsup{\mathcal E}_n(\psi))\in\{0,1\}$ \cite{Gallzeroone}.
  \item[Subhomogeneity] For any $t\geq 1$, $\lambda(\limsup{\mathcal E}_n(t\psi))\leq t\lambda(\limsup{\mathcal E}_n(\psi))$ \cite{Li2}.
\end{description}
\end{theorem}
We  need another  lemma.
\begin{lemma}\label{apple}
 Let $\mathcal{D}=\{n_k\}$ be a pseudo-absolute value sequence. Then
 \begin{equation}\label{18equ2}
    \sum _{n_k\leq n} n_k\log\frac{n}{n_k}\leq Cn,
 \end{equation}
and
\begin{equation}\label{18equ3}
  \sum_{2^{2^N}\leq n_k\leq 2^{2^{N+1}}}\frac{1}{\log n_k}=O(1).
\end{equation}
\end{lemma}
\begin{proof}
Since $\{n_k\}$ is a pseudo-absolute value sequence,
  there exists at most one $n_k$ such that $2^j\leq n_k<2^{j+1}$.
Thus
\begin{eqnarray*}
   \sum _{n_k\leq n} n_k\log\frac{n}{n_k} &\leq&   \sum_{j=0}^{\log_2 n}  \sum_{2^j\leq n_k<2^{j+1}} n_k\log \frac{n}{n_k}\\
 &\leq&   \sum_{j=0}^{\log_2 n}   2^{j+1}\log \frac{n}{2^j} \\
   &\leq&  Cn.
\end{eqnarray*}
This proves \eqref{18equ2}.

Similarly, we have

\begin{eqnarray*}
   \sum_{n=2^{2^N}+1}^{2^{2^{N+1}}}\frac{1}{\log n_k} &\leq&   \sum_{j=2^N}^{2^{N+1}}  \sum_{2^j\leq n_k<2^{j+1}} \frac{1}{\log n_k}\\
 &\leq&  O(1)  \sum_{j=2^N}^{2^{N+1}}   \frac{1}{j} \\
   &=& O(1).
\end{eqnarray*}
We finish the proof.
\end{proof}
%
%
After the preparations, we can prove the case $\epsilon=0$ of Theorem \ref{thmnewhar}.
\begin{theorem}\label{Lithm2}
Let $\mathcal{D}=\{n_k\}$ be a pseudo-absolute value sequence and  $\mathfrak{M}(n)$ be given by \eqref{GMhar}.
Then
for almost every $\alpha\in\R$
	\begin{equation*}
		\liminf_{n\to\infty}n \mathfrak{M}(n)(\log n)  |{n}|_{\mathcal{D} }||{n\alpha}||^\prime =0.
	\end{equation*}
\end{theorem}
\begin{proof}
Without loss of generality, assume $\alpha\in[0,1)$.
 Let $$\psi_0(n) =\frac{1}{  |{n}|_{\mathcal{D} } n\mathfrak{M}(n)(\log n)},$$
 and
 \begin{equation*}
    \psi(n) =\frac{1}{    n\mathfrak{M}(n)(\log n)}.
 \end{equation*}
It suffices to show that there exists some $c>0$ such that
 \begin{equation}\label{Gli2last}
G_N=\sum_{n=2^{2^N}+1}^{2^{2^{N+1}}}\frac{\psi_0(n)\varphi(n)}{n}>c
\end{equation}
for  $N\in\N$.
Indeed, if \eqref{Gli2last} holds, then for any $\varepsilon>0$, there exists some $C>0$ such that
\begin{equation*}
    \sum_{n=2^{2^N}+1}^{2^{2^{N+1}}}\frac{C\varepsilon\psi_0(n)\varphi(n)}{n}\geq 3 \text{ for all }   N.
\end{equation*}
Applying   Lemma \ref{Har} (letting $\psi=C\varepsilon\psi_0$),  one has
\begin{equation}\label{sub}
    \lambda(\limsup{\mathcal E}_n(C\varepsilon\psi_0))=1.
\end{equation}
Applying Theorem \ref{zeroone} (Subhomogeneity) to \eqref{sub},
we obtain
\begin{equation*}
    \lambda({\limsup\mathcal E}_n(\varepsilon\psi_0))\geq \frac{1}{C}.
\end{equation*}
By zero-one law of Theorem \ref{zeroone},  we have
\begin{equation*}
    \lambda(\limsup{\mathcal E}_n(\varepsilon\psi_0))=1.
\end{equation*}
Thus for any $\varepsilon>0$, we have that
 for almost every $\alpha$, the inequality
\begin{equation*}
    |n\alpha -p|\leq \varepsilon\psi _0(n)
\end{equation*}
has infinitely  many coprime solutions $(n,p)\in \N\times\Z$.
This implies that  for almost every $\alpha\in\R$
	\begin{equation*}
		\liminf_{n\to\infty}n \mathfrak{M}(n)(\log n)  |{n}|_{\mathcal{D} }||{n\alpha}||^\prime =0.
	\end{equation*}
Now we focus on the proof of \eqref{Gli2last}.

As usual, we have
\begin{equation*}
    \sum_{n=2^{2^N}+1}^{2^{2^{N+1}}}\frac{\varphi(n)\psi(n)}{n|{n}|_{\mathcal{D}}}
\end{equation*}
\begin{equation}\label{equ11liwen1}
	   =     \sum_{n=2^{2^N}+1}^{2^{2^{N+1}}}
	\left(\psi(n)-\psi(n+1) \right) \sum_{j=2^{2^N}+1}^n \frac{\varphi(j)}{j|{j}|_{\mathcal{D}}}
	+ \psi(2^{2^{N+1}}+1) \sum_{j=2^{2^N}+1}^{2^{2^{N+1}}} \frac{\varphi(j)}{j|{j}|_{\mathcal{D}}}.
\end{equation}
Direct computation yields to
\begin{align}
\sum_{j=2^{2^N}+1}^n \frac{\varphi(j)}{ j|{j}|_{\mathcal{D}}}&=\sum_{k:1\leq n_k\leq n}\sum_{\substack{j=2^{2^N}+1\\n_{k } |j,~n_{k+1}\nmid j}}^n\frac{\varphi(j)}{ j|{j}|_{\mathcal{D}}}\nonumber\\
&=\sum_{n_k\leq n}\sum_{\substack{ \frac{2^{2^N}+1}{n_k}\le j\le \frac{n}{n_k}\\d_{k +1} \nmid j}} \frac{\varphi(n_kj)}{j}\nonumber\\
&\ge\sum_{n_k\leq n}\varphi(n_k)\sum_{\substack{ \frac{2^{2^N}+1}{n_k}\le j\le \frac{n}{n_k}\\d_{k +1} \nmid j}} \frac{\varphi(j)}{j}\nonumber\\
&\ge \frac{4}{\pi^2} \sum_{n_k\leq n} \varphi(n_k )\max\{0,\frac{n-2^{2^N}}{n_k}-O\left(\log(\frac{n}{n_k})\right)\}\nonumber\\
&\ge \frac{4}{\pi^2} \sum_{n_k\leq n} \frac{\varphi(n_k )}{n_k}\left( ( n-2^{2^N})-O(n_k\log\frac{n}{n_k})\right)\nonumber\\
&\ge \frac{4}{\pi^2} \mathfrak{M}(n)( n-2^{2^N})- \sum _{n_k\leq n}O\left(n_k\log\frac{n}{n_k}\right),\label{equ1final}
\end{align}
where  the second  inequality holds by Lemma \ref{keysection2li}.

By the definition of $\psi(n)$, we have  for $n\neq n_k$,
\begin{equation}\label{18equ5}
    \psi(n)-\psi(n+1)=\frac{O(1)}{n^2\mathfrak{M}(n)\log n},
\end{equation}
and
\begin{equation}\label{18equ6}
    \psi(n_k)-\psi(n_k+1)=\frac{O(1)}{n_k\mathfrak{M}^2(n_k)\log n_k}.
\end{equation}
By \eqref{18equ2}, \eqref{18equ5} and \eqref{18equ6},
one has
\begin{equation*}
    \sum_{n=2^{2^N}+1}^{2^{2^{N+1}}}
	\left(\psi(n)-\psi(n+1) \right) \sum _{n_k\leq n} n_k\log\frac{n}{n_k}
	+ \psi(2^{2^{N+1}}+1)  \sum _{n_k\leq 2^{2^{N+1}}} n_k\log\frac{2^{2^{N+1}}}{n_k}
\end{equation*}
\begin{eqnarray}
    &\leq&  \sum_{n=2^{2^N}+1}^{2^{2^{N+1}}}  \frac{O(1)}{n\mathfrak{M}(n)\log n} + \sum_{2^{2^N}\leq n_k\leq 2^{2^{N+1}}}\frac{O(1)}{\mathfrak{M}^2(n_k)\log n_k}+ \frac{O(1)}{\mathfrak{M}(2^{2^{N+1}})}\nonumber\\
 &\leq& \frac{O(1)}{\mathfrak{M}^2(2^{2^{N}})} +\frac{O(1)}{\mathfrak{M}(2^{2^{N+1}})}+\frac{ O(1)}{\mathfrak{M}(2^{2^{N}})} \sum_{n=2^{2^N}+1}^{2^{2^{N+1}}}  \frac{1}{n\log n} \nonumber\\
  &=&   \frac{ O(1)}{\mathfrak{M}(2^{2^{N}})},\label{Wen1}
\end{eqnarray}
where the second inequality holds by \eqref{18equ3} and the third inequality holds because of  ($a=2^{2^{N}}$ and $b=2^{2^{N+1}}$)
\begin{equation}\label{18equ7}
 \sum_{a}^{b}\frac{1}{n\log n}\asymp \int_a^b \frac{dx}{x\log x}=\log\log b-\log\log a \text { for any } b>a>1.
\end{equation}


Putting \eqref{equ1final} and \eqref{Wen1} into \eqref{equ11liwen1},
we obtain
   \begin{eqnarray*}
      \sum_{n=2^{2^N}+1}^{2^{2^{N+1}}}\frac{\varphi(n)\psi(n)}{n|{n}|_{\mathcal{D} }} &\geq&\sum_{n=2^{2^N}+1}^{2^{2^{N+1}}}  c\left(\frac{1}{n\log n\mathfrak{M}(n)}-\frac{1}{(n+1)\log(n+1)\mathfrak{M}(n+1)} \right)\mathfrak{M}(n)(n-2^{2^N})-  \frac{ O(1)}{\mathfrak{M}(2^{2^{N}})}\\
       &\geq&
      \sum_{n=2^{(2^N+4)}}^{2^{2^{N+1}}} \frac{c}{2}\left(\frac{1}{n\log n\mathfrak{M}(n)}-\frac{1}{(n+1)\log(n+1)\mathfrak{M}(n+1)} \right)n\mathfrak{M}(n)-  \frac{ O(1)}{\mathfrak{M}(2^{2^{N}})}\\
     &\geq&  c\sum_{n=2^{(2^N+4)}}^{2^{2^{N+1}}}\frac{1}{n\log n}- \frac{ O(1)}{\mathfrak{M}(2^{2^{N}})}.
   \end{eqnarray*}
Using \eqref{18equ7} again,
\begin{equation*}
    \sum_{n=2^{(2^N+4)}}^{2^{2^{N+1}}}\frac{1}{n\log n}\asymp 1.
\end{equation*}
This yields that for some $c>0$,
\begin{equation*}
    G_N\geq c.
\end{equation*}
We finish the proof.
\end{proof}

 \section*{Acknowledgments}
I would like to thank Svetlana
Jitomirskaya for   introducing to me the Littlewood  conjecture and  comments on earlier versions of the manuscript. I also thank   the  anonymous referee for careful reading of the manuscript that has led to an important improvement.
 The author  was supported by the AMS-Simons Travel Grant 2016-2018 and  NSF DMS-1700314. This research was also
partially supported by NSF DMS-1401204.



\footnotesize
 \begin{thebibliography}{10}

\bibitem{Bad}
D.~A. Badziahin.
\newblock On multiplicatively badly approximable numbers.
\newblock {\em Mathematika}, 59(1):31--55, 2013.

\bibitem{beresnevich2013duffin}
V.~Beresnevich, G.~Harman, A.~Haynes, and S.~Velani.
\newblock The {D}uffin-{S}chaeffer conjecture with extra divergence {II}.
\newblock {\em Mathematische Zeitschrift}, 275(1-2):127--133, 2013.

\bibitem{bourgain2009some}
J.~Bourgain, E.~Lindenstrauss, P.~Michel, and A.~Venkatesh.
\newblock Some effective results for$\times$ a$\times$ b.
\newblock {\em Ergodic Theory and Dynamical Systems}, 29(06):1705--1722, 2009.

\bibitem{bugeaud2014around}
Y.~Bugeaud.
\newblock Around the {L}ittlewood conjecture in {D}iophantine approximation.
\newblock {\em Publications math{\'e}matiques de Besan{\c{c}}on}, (1):5--18,
  2014.

\bibitem{bugeaud2011metric}
Y.~Bugeaud, A.~Haynes, and S.~Velani.
\newblock Metric considerations concerning the mixed {L}ittlewood conjecture.
\newblock {\em International Journal of Number Theory}, 7(03):593--609, 2011.

\bibitem{bugeaud2011badly}
Y.~Bugeaud and N.~Moshchevitin.
\newblock {B}adly approximable numbers and {L}ittlewood-type problems.
\newblock In {\em Mathematical Proceedings of the Cambridge Philosophical
  Society}, volume 150, pages 215--226. Cambridge Univ Press, 2011.

\bibitem{de2004problemes}
B.~de~Mathan and O.~Teuli{\'e}.
\newblock Problemes diophantiens simultan{\'e}s.
\newblock {\em Monatshefte f{\"u}r Mathematik}, 143(3):229--245, 2004.

\bibitem{duffin1941khintchine}
R.~Duffin and A.~Schaeffer.
\newblock Khintchine’s problem in metric {D}iophantine approximation.
\newblock {\em Duke Mathematical Journal}, 8(2):243--255, 1941.

\bibitem{einsiedler2006invariant}
M.~Einsiedler, A.~Katok, and E.~Lindenstrauss.
\newblock Invariant measures and the set of exceptions to {L}ittlewood's
  conjecture.
\newblock {\em Annals of mathematics}, pages 513--560, 2006.

\bibitem{einsiedler2007measure}
M.~Einsiedler and D.~Kleinbock.
\newblock Measure rigidity and $p$-adic {L}ittlewood-type problems.
\newblock {\em Compositio Mathematica}, 143(03):689--702, 2007.

\bibitem{Furstenberg}
H.~Furstenberg.
\newblock Disjointness in ergodic theory, minimal sets, and a problem in
  {D}iophantine approximation.
\newblock {\em Math. Systems Theory}, 1:1--49, 1967.

\bibitem{Gallzeroone}
P.~Gallagher.
\newblock Approximation by reduced fractions.
\newblock {\em J. Math. Soc. Japan}, 13:342--345, 1961.

\bibitem{Gal}
P.~Gallagher.
\newblock Metric simultaneous diophantine approximation.
\newblock {\em J. London Math. Soc.}, 37:387--390, 1962.

\bibitem{harrap2013mixed}
S.~Harrap and A.~Haynes.
\newblock The mixed {L}ittlewood conjecture for pseudo-absolute values.
\newblock {\em Mathematische Annalen}, 357(3):941--960, 2013.

\bibitem{haynes2012duffin}
A.~K. Haynes, A.~D. Pollington, and S.~L. Velani.
\newblock The {D}uffin-{S}chaeffer conjecture with extra divergence.
\newblock {\em Mathematische Annalen}, 353(2):259--273, 2012.

\bibitem{Li1}
L.~Li.
\newblock A note on the {D}uffin-{S}chaeffer conjecture.
\newblock {\em Unif. Distrib. Theory}, 8(2):151--156, 2013.

\bibitem{Li2}
L.~Li.
\newblock The {D}uffin--{S}chaeffer-type conjectures in various local fields.
\newblock {\em Mathematika}, 62(3):753--800, 2016.

\bibitem{Sch}
Y.~Peres and W.~Schlag.
\newblock Two {E}rd{\H{o}}s problems on lacunary sequences: chromatic number
  and {D}iophantine approximation.
\newblock {\em Bull. Lond. Math. Soc.}, 42(2):295--300, 2010.

\bibitem{PV}
A.~D. Pollington and R.~C. Vaughan.
\newblock The {$k$}-dimensional {D}uffin and {S}chaeffer conjecture.
\newblock {\em Mathematika}, 37(2):190--200, 1990.

\bibitem{Stra}
O.~Strauch.
\newblock Duffin-{S}chaeffer conjecture and some new types of real sequences.
\newblock {\em Acta Math. Univ. Comenian.}, 40(41):233--265, 1982.

\end{thebibliography}
\end{document}